\documentclass[11pt,reqno]{amsart}
\usepackage{amsthm,amssymb,amsmath}
\usepackage{xcolor}

\newtheorem{theorem}{Theorem}
\newtheorem*{theorem*}{Theorem}
\newtheorem{lemma}{Lemma}
\newtheorem*{lemmaA}{Lemma A2}
\newtheorem*{lemmaB}{Lemma A1}

\newtheorem{claim}{Claim}

\theoremstyle{definition}

\newtheorem*{definition*}{\sc Definition}

\newtheorem{remark}{\bf Remark}
\newtheorem*{remark*}{\bf Remark}

\newtheorem*{example*}{\bf Example}

\newcommand{\loc}{{\rm loc}}

\newcommand{\dist}{\mbox{dist}}

\newcommand{\const}{{\rm const}}
\newcommand{\cl}{{\rm clos}}

\newcommand{\clos}{{\rm clos}}

\newcommand{\thistheoremname}{}

\newtheorem*{genericthm*}{\thistheoremname}
\newenvironment{namedthm*}[1]
  {\renewcommand{\thistheoremname}{#1}%
   \begin{genericthm*}}
  {\end{genericthm*}}

\expandafter\def\expandafter\normalsize\expandafter{%
    \normalsize
    \setlength\abovedisplayshortskip{8pt}
    \setlength\belowdisplayshortskip{8pt}
}

\begin{document}

\title{Stochastic differential equations with singular (form-bounded) drift}

\author{D.\,Kinzebulatov and Yu.\,A.\,Sem\"{e}nov} 

\address{Universit\'{e} Laval, D\'{e}partement de math\'{e}matiques et de statistique, pavillon Alexandre-Vachon 1045, av. de la M\'{e}decine, Qu\'{e}bec, PQ, G1V 0A6, Canada}

\email{damir.kinzebulatov@mat.ulaval.ca}

\thanks{D.K.\,is supported by grants of NSERC and FRQNT}

\address{University of Toronto, Department of Mathematics, 40 St.\,George Str., Toronto, ON, M5S 2E4, Canada}

\email{semenov.yu.a@gmail.com}



\begin{abstract}
We consider the problem of constructing weak solutions to the It\^{o} and to the Stratonovich stochastic differential equations having critical-order singularities in the drift and critical-order discontinuities in the dispersion matrix. 
\end{abstract}

\maketitle

\textbf{1.~}We consider the problem of constructing  weak solutions to the It\^{o} stochastic differential equation (SDE)
\begin{equation}
\label{sde1}
\tag{$I$}
X(t)=x - \int_0^t b(X(s))ds + \sqrt{2} \int_0^t \sigma(X(s))dW(s), \quad x \in \mathbb R^d,
\end{equation}
($d \geq 3$) and to the Stratonovich SDE
\begin{equation}
\label{sde2}
\tag{$S$}
X(t)=x - \int_0^t b(X(s))ds  + \sqrt{2} \int_0^t \sigma(X(s)) \circ dW(s), \quad x \in \mathbb R^d,
\end{equation}
under the following assumptions on the drift $b:\mathbb R^d \rightarrow \mathbb R^d$ and the dispersion matrix 
$\sigma \in L^\infty(\mathbb R^d,\mathbb R^d \otimes \mathbb R^d)$:

1)  $b$ is form-bounded, i.e.\,$|b|^2 \in L^2_{\loc} \equiv L^2_{\loc}(\mathbb R^d)$ and 
$$
\||b|(\lambda -\Delta)^{-\frac{1}{2}}\|_{2\rightarrow 2}\leqslant \sqrt{\delta}
$$
for $\delta>0$ and $\lambda=\lambda_\delta>0$ 
(write $b \in \mathbf{F}_\delta$). Here $\|\cdot\|_{2 \rightarrow 2}:=\|\cdot\|_{L^2 \rightarrow L^2}$.

The class $\mathbf{F}_\delta$ contains vector fields in $[L^p + L^\infty]^d$ , $p>d$ (by H\"{o}lder's inequality) and in $[L^d + L^\infty]^d$ (by Sobolev's inequality) with the relative bound $\delta$ that can be chosen arbitrarily small. The class $\mathbf{F}_\delta$ also contains vector fields having critical-order singularities, such as $b(x)=\sqrt{\delta}\frac{d-2}{2}|x|^{-2}x$ (by Hardy's inequality) or, more generally, vector fields in the weak $L^d$ class (by Strichartz' inequality \cite{KPS}), the Campanato-Morrey class or the Chang-Wilson-Wolff class \cite{CWW}, with $\delta$ depending on the respective norm of the vector field in these classes. It is clear that $b_1 \in \mathbf{F}_{\delta_1}$, $b_2 \in \mathbf{F}_{\delta_2}$ $\Rightarrow$ $b_1+b_2 \in \mathbf{F}_\delta$, $\sqrt{\delta}=\sqrt{\delta_1}+\sqrt{\delta_2}$. We refer to \cite{KiS} for a more detailed discussion on the class $\mathbf{F}_\delta$.

2) $a:=\sigma \sigma^{\intercal} \geq \nu I$, $\nu>0$, and
\begin{equation*}
\nabla_r a_{\cdot \ell}\in \mathbf{F}_{\gamma_{r\ell}} \quad ( 1 \leq r,\ell \leq d)
\end{equation*}
for some $\gamma_{r\ell}>0$. 

By 1), a matrix $a$ with entries in $W^{1,d}$ satisfies 2) with $\gamma_{r\ell}$ that can be chosen arbitrarily small.
The model example of a matrix $a$ satisfying 2) and having a critical discontinuity is $$a(x)=I+c\frac{x\otimes x}{|x|^2}, \quad c>-1.$$
Another example is $$a(x)=I+c (\sin \log(|x|))^2 e \otimes e, \quad e \in \mathbb R^d, |e|=1,$$ or, more generally, a sum of these two matrices with their points of discontinuity constituting e.g.\,a dense subset of $\mathbb R^d$. 

The problem of existence of a (unique in law) weak solution to the It\^{o} SDE  \eqref{sde1} with a locally unbounded general $b$ (i.e.\,not necessarily differentiable, radial or having other additional structure) is of fundamental importance, and has been thoroughly studied in the literature. 
The first principal result is due to N.\,I.\,Portenko \cite{P}: if $a$ is H\"{o}lder continuous and $b \in [L^p+L^\infty]^d$, $p>d$, then there exists a unique in law weak solution to \eqref{sde1}.
This result has been strengthened in the case $a=I$ in \cite{BC} for $b$ in the Kato class $\mathbf{K}^{d+1}_0$, and in \cite{KiS4} for $b$ is in the class of weakly form-bounded vector fields $\mathbf{F}_\delta^{\scriptscriptstyle 1/2}$  (see remark below concerning the uniqueness). (The class $\mathbf{F}_\delta^{\scriptscriptstyle 1/2}=\{|b| \in L^1_{\loc}: \||b|^{\frac{1}{2}}(\lambda-\Delta)^{-\frac{1}{4}}\|_{2 \rightarrow 2} \leq \delta\}$ contains both the Kato class $\mathbf{K}^{d+1}_\delta=\{|b| \in L^1_{\loc}: \||b|(\lambda-\Delta)^{-\frac{1}{2}}\|_{1 \rightarrow 1}\leq \delta\}$ and $\mathbf{F}_\delta$ as proper subclasses. It also contains the sums of the vector fields in these two classes.) Since already $\mathbf{K}^{d+1}_0:=\cap_{\delta>0}\mathbf{K}^{d+1}_\delta$ contains, for every $\varepsilon>0$, vector fields $b \not\in L^{1+\varepsilon}_{\loc}$, one can not appeal to the Girsanov transform in order to construct a weak solution of \eqref{sde1}. We note that $\mathbf{K}^{d+1}_0 - \mathbf{F}_\delta \neq \varnothing$, $\mathbf{F}_\delta - \mathbf{K}^{d+1}_{\delta_1} \neq \varnothing$ (in fact, already $[L^d + L^\infty]^d \not\subset \mathbf{K}^{d+1}_{\delta_1}$).

In Theorems \ref{mainthm} and \ref{mainthm2} below we prove that, under appropriate assumptions on relative bounds $\delta$ and $\gamma_{r\ell}$ ($1 \leq r,\ell \leq d$), the SDEs \eqref{sde1} and \eqref{sde2} have weak solutions, for every $x \in \mathbb R^d$, which determine a Feller semigroup on $C_\infty:=\{g \in C(\mathbb R^d): \lim_{x \rightarrow \infty}g(x)=0\}$ (with the $\sup$-norm). The latter is, in fact, the starting object in our approach.

The dependence of the solvability of \eqref{sde1}, \eqref{sde2} on the values 
of relative bounds has fundamental nature. For example, consider the vector field ($d \geq 3$) $$b(x):= \sqrt{\delta} \frac{d-2}{2} |x|^{-2} x \in \mathbf{F}_\delta.$$
If $\sqrt{\delta}<1 \wedge \frac{2}{d-2}$, then by Theorem \ref{mainthm} below the SDE
\[
X(t) =  - \int_0^t b(X(s))ds + \sqrt{2} W(t), \quad t \geq 0.
\]
has a weak solution.
If $\sqrt{\delta} \geq \frac{2d}{d-2}$, then an elementary argument shows that the equation does not have a weak solution, cf.\,\cite[Example 1]{KiS4}. In this sense, Theorem \ref{mainthm} covers critical-order singularities of $b$.

The central analytic object in our approach is $\Lambda_q(a,b)$, an operator realization  of the formal operator $-\nabla \cdot a \cdot \nabla + b \cdot \nabla$  in $L^q$ (we write $\Lambda_q(a,b) \supset -\nabla \cdot a \cdot \nabla + b \cdot \nabla$), an associated with it Feller semigroup on $C_\infty$ and the $W^{1,p}$ estimates on solutions of the corresponding elliptic equation. By 2), the vector field $\nabla a$ defined by $(\nabla a)^k := \sum_{i=1}^d (\nabla_i a_{ik})$ is in the class $\mathbf{F}_{\delta_a}$ with $\delta_a \leq \gamma:=\sum_{r,\ell=1}^d \gamma_{r\ell}$. Thus, $\Lambda(a,\nabla a +b) \supset -a \cdot \nabla^2 + b \cdot \nabla$ is well defined. We will show that the probability measures determined by the Feller semigroup associated to $\Lambda(a,\nabla a +b)$ admit description as weak solutions to \eqref{sde1}. 
(Since we only require that $\nabla a + b$ is in $\mathbf{F}_\delta$, we can handle diffusion matrices  having critical discontinuities; on the other hand, if we would require more, e.g.\,$\nabla_r a_{i \ell} \in L^p + L^\infty$ for some $p>d$, then by the Sobolev Embedding Theorem $a$ would be H\"{o}lder continuous, and we would end up in the assumptions of \cite{P}.)

We note that the results concerning \eqref{sde1} that impose various conditions on the derivatives of $a_{k\ell}$ already appeared in the literature, see e.g.\,\cite{ZZ}, see also references therein.

The assumptions 1), 2) destroy the two-sided Gaussian bounds on the heat kernel of $-\nabla \cdot a \nabla + b \cdot \nabla$, $-a \cdot \nabla^2 + b \cdot \nabla$  (this is already apparent if $a=I$, $b(x)=\pm\frac{d-2}{2}\sqrt{\delta}|x|^{-2}x$).

Concerning the Stratonovich SDE \eqref{sde2}, 
instead of 2) we require:

\smallskip

2') 
$
\nabla_r \sigma_{\cdot j} \in \mathbf{F}_{\delta_{rj}}$ for some $\delta_{rj}>0$ ($1 \leq r,j \leq d).
$

\smallskip

\noindent We re-write \eqref{sde2} as
\begin{equation}
\label{sde3}
\tag{$S'$}
X(t)=x - \int_0^t b(X(s))ds + \int_0^t c(X(s))ds + \sqrt{2} \int_0^t \sigma(X(s))dW(s), \qquad x \in \mathbb R^d,
\end{equation}
where
\begin{equation}
\label{c}
c:=(c^i)_{i=1}^d, \quad c^i(x):=\frac{1}{\sqrt{2}}\sum_{r,j=1}^d  (\nabla_r \sigma_{ij})\sigma_{rj}.
\end{equation}
Then, by 2'), $$c \in \mathbf{F}_{\delta_c}, \quad \delta_c \leq \frac{1}{2}\|\sigma\|^2_\infty\sum_{r,j=1}^d \delta_{rj}$$ (here $\|\sigma\|_\infty=\|(\sum_{r,j=1}^d \sigma_{rj}^2)^{\frac{1}{2}}\|_\infty$). We note that 2') yields 2). Indeed, 
$$\nabla_r a_{\cdot \ell}\in \mathbf{F}_{\gamma_{r\ell}}, \quad \gamma_{r\ell}\leq \big[\|\sigma_{\cdot\ell}\|_\infty(\sum_{j=1}^d\delta_{rj})^{\frac{1}{2}} + \|\sigma\|_\infty \delta_{r\ell}^{\frac{1}{2}}\big]^2.$$
Thus, we put \eqref{sde2} in the It\^{o} form, however, without losing the class of singularities of the drift or the class of discontinuities of the dispersion matrix. From the analytic point of view, imposing conditions on $\nabla_r \sigma_{ij}$ seems to be pertinent to the subject matter since it provides an operator behind \eqref{sde2}.

We prove that the weak solution to \eqref{sde1} or \eqref{sde2} is unique among all weak solutions that can be constructed using reasonable approximations of $a$, $b$, i.e.\,the ones that keep the values of relative bounds intact, see remark \ref{unique_rem} below.
We do not prove the uniqueness is law. (In this regard, we note that, under the assumptions 1), 2), in general $|\nabla u| \not \in L^\infty$, $u=(\mu + \Lambda_q(a,\nabla a + b))^{-1}f$, even if $f \in C_c^\infty$.)
However, in our construction the weak solutions to \eqref{sde1}, \eqref{sde2} are determined from the very beginning by a Feller semigroup, and so the associated process is strong Markov. The lack of the uniqueness in law, arguably, does not have decisive importance for completeness of the result.

\medskip

\medskip

\textbf{2.~}The following analytic results are crucial for what follows. Without loss of generality, we assume from now on that $a \geq I$. 

 Let $a$, $b$ satisfy conditions 1), 2). Assume that the relative bounds $\delta$, $\gamma$, $\delta_a$ satisfy, for some $q > 2 \vee (d-2)$,
\begin{equation}
\label{cond0}
\left\{
\begin{array}{l}
1-\frac{q}{4}(\sqrt{\gamma} + \|a-I\|_\infty \sqrt{\delta}) > 0, \\
(q-1)\big(1-\frac{q\sqrt{\gamma}}{2}\big) - (\sqrt{\delta}\sqrt{\delta_a} + \delta)\frac{q^2}{4} - (q-2)\frac{q\sqrt{\delta}}{2} -\|a-I\|_\infty \frac{q\sqrt{\delta}}{2}>0.
\end{array}
\right.
\end{equation}
(For example, \eqref{cond0} is evidently satisfied for all $\delta$, $\gamma$, $\delta_a$ sufficiently small. If $\gamma=0$, then \eqref{cond0} reduces to $\delta<1 \wedge (\frac{2}{d-2})^2$.)
Then, by  \cite[Theorem 2]{KiS3}, there exists an operator realization $\Lambda_q(a,b)$ of the  formal differential operator $-\nabla \cdot a \cdot \nabla + b \cdot \nabla$ in $L^q$  as the (minus) generator of a positivity preserving $L^\infty$ contraction quasi contraction $C_0$ semigroup $e^{-t\Lambda_q(a,b)}$,
\begin{equation}
\label{conv1}
e^{-t\Lambda_q(a,b)}:=s{\mbox-}L^q{\mbox-}\lim e^{-t\Lambda_q(a_n,b_n)} \quad (\text{loc.\,uniformly in $t \geq 0$}),
\end{equation}
where $\Lambda_q(a_n,b_n):=-\nabla \cdot a_n \cdot \nabla + b_n \cdot \nabla$, $D(\Lambda_q(a_n,b_n))=W^{2,p}$,
$b_n := e^{\varepsilon_n\Delta} (\mathbf 1_n b)$, $\mathbf 1_n$ is the indicator of  $\{x \in \mathbb R^d \mid  \; |x| \leq n,  |b(x)| \leq n \}$, $\varepsilon_n \downarrow 0$, 
$
a_n:=I+e^{\epsilon_n\Delta}\big(\eta_n (a-I)\big),
$   
$$
\eta_n(x):=\left\{
\begin{array}{ll}
1, & \text{ if } |x|< n, \\
n+1-|x|, & \text{ if } n \leq |x| \leq n+1, \qquad (x \in \mathbb R^d), \qquad \epsilon_n \downarrow 0,\\
0, & \text{ if } |x|>n+1,
\end{array}
\right.
$$
(see remark \ref{approx_rem} below), such that for $u:=(\mu+\Lambda_q(a,b))^{-1}f$, $\mu>\mu_0$, $f \in L^q$,
\begin{equation}
\tag{$\star
$}
\label{reg_star}
\begin{array}{c}
\|\nabla u \|_{q} \leq  K_1(\mu-\mu_0)^{-\frac{1}{2}} \|f\|_q, \\[2mm]
\|\nabla u \|_{\frac{qd}{d-2}} \leq  K_2 (\mu-\mu_0)^{\frac{1}{q}-\frac{1}{2}} \|f\|_q
\end{array}
\end{equation}
where the constants $\mu_0>0$ and $K_i<\infty$ ($i=1,2$) depend only on $d,q,c,\delta,\gamma$. By \eqref{reg_star} and the Sobolev Embedding Theorem, $u \in C^{0,\alpha}$, $\alpha=1-\frac{d-2}{p}$. (See remark \ref{proof_rem} below.)

The second estimate in \eqref{reg_star} allows us to run an iteration procedure $L^p \rightarrow L^\infty$, which, combined with  \eqref{conv1}, allows to construct a positivity preserving contraction $C_0$ semigroup on $C_\infty$ (Feller semigroup) by the formula
\begin{equation}
\label{conv_c0}
e^{-t\Lambda_{C_\infty}(a,b)}:=s{\mbox-}C_\infty{\mbox-}\lim e^{-t\Lambda_{C_\infty}(a_n,b_n)} \quad (\text{loc.\,uniformly in $t \geq 0$}),
\end{equation}
where $\Lambda_{C_\infty}(a_n,b_n):=-\nabla \cdot a_n \cdot \nabla + b_n \cdot \nabla$, $D(\Lambda_{C_\infty}(a_n,b_n)):=(1-\Delta)^{-1}C_\infty$ \cite[Theorem 3]{KiS3}. 
(The reason we first work in $L^q$, and not directly in $C_\infty$, is simple: $L^q$ has a (locally) weaker topology, so it is much easier to prove convergence there.)

By \eqref{conv1} and \eqref{conv_c0},
\begin{equation}
\label{consist0}
(\mu+\Lambda_{C_\infty}(a,b))^{-1}=\bigl((\mu+\Lambda_{q}(a,b))^{-1} \upharpoonright L^q \cap C_\infty \bigr)^{\clos}_{C_\infty \rightarrow C_\infty}, \quad \mu > \mu_0.
\end{equation}
In view of \eqref{consist0}, \eqref{reg_star} and the Sobolev Embedding Theorem,
\begin{equation}
\label{SF0}
 \bigl(e^{-t\Lambda_{C_\infty}(a,b)} \upharpoonright L^q \cap C_\infty \bigr)^{\cl}_{L^q \rightarrow C_\infty} \in \mathcal B(L^q,C_\infty), \qquad \; t>0.
\end{equation}

\begin{remark}
It is clear that $C_c^\infty \not\subset D(\Lambda_{C_\infty}(I,b))$ for $b \in [L^\infty]^d - [C_b]^d$.
In fact, an attempt to find a complete description of $D(\Lambda_{C_\infty}(a,b))$ in the elementary terms for a general $b \in \mathbf{F}_\delta$, even if $a=I$, is rather hopeless.
\end{remark}

\begin{remark}
\label{approx_rem}
Since our assumptions on $\delta$, $\gamma$ and $\delta_a$ involve only strict inequalities, we can and will choose $\epsilon_n, \varepsilon_n \downarrow 0$ in the definition of $a_n$, $b_n$ so that 
$$
\nabla_r (a_n)_{\cdot \ell}\in \mathbf{F}_{\tilde{\gamma}_{r\ell}} \;\;(1 \leq r,\ell \leq d), \quad
\nabla a_n \in \mathbf{F}_{\tilde{\delta}_a}, \quad b_n \in \mathbf{F}_{\tilde{\delta}} 
$$
with relative bounds $\tilde{\delta}$, $\tilde{\gamma}_{rk}$, $\tilde{\delta_a}$ satisfying \eqref{cond0}, and with $\lambda \neq \lambda(n)$.

In what follows, without loss of generality, $\tilde{\delta}=\delta$, $\tilde{\gamma}=\gamma$, $\tilde{\delta}_a=\delta_a$.
\end{remark}

\medskip

\textbf{3.~}We now state the main results of the paper. We consider first the It\^{o} SDE \eqref{sde1}.
The corresponding analytic object is $\Lambda_q(a,\nabla a + b)$, an operator realization of $-a \cdot \nabla^2 + b \cdot \nabla$ in $L^q$, see the previous section, where we assume that the condition \eqref{cond0} is satisfied with $\delta$ replaced by $\delta_a + \delta$.
Then $\nabla a_n + b_n \in \mathbf{F}_{\delta_a+\delta}$ with $\lambda \neq \lambda(n)$, and the limit
\begin{equation}
\label{conv_c}
e^{-t\Lambda_{C_\infty}(a,\nabla a + b)}:=s{\mbox-}C_\infty{\mbox-}\lim e^{-t\Lambda_{C_\infty}(a_n,\nabla a_n + b_n)} \quad (\text{loc.\,uniformly in $t \geq 0$}),
\end{equation}
where $\Lambda_{C_\infty}(a_n,\nabla a_n + b_n):=-a_n \cdot \nabla^2 + b_n \cdot \nabla$, $D(\Lambda_{C_\infty}(a_n,\nabla a_n + b_n)):=(1-\Delta)^{-1}C_\infty$, exists and determines Feller semigroup on $C_\infty$. 
By \eqref{conv_c0}, \eqref{consist0}, 
\begin{equation}
\label{consist}
(\mu+\Lambda_{C_\infty}(a,\nabla a + b))^{-1}=\bigl((\mu+\Lambda_{q}(a,\nabla a + b))^{-1} \upharpoonright L^q \cap C_\infty \bigr)^{\clos}_{C_\infty \rightarrow C_\infty}, \quad \mu > \mu_0.
\end{equation}
\begin{equation}
\label{SF}
 \bigl(e^{-t\Lambda_{C_\infty}(a,\nabla a + b)} \upharpoonright L^q \cap C_\infty \bigr)^{\cl}_{L^q \rightarrow C_\infty} \in \mathcal B(L^q,C_\infty), \qquad \; t>0.
\end{equation}
Denote: $\bar{\mathbb{R}}^d:=\mathbb R^d \cup \{\infty\}$ is the one-point compactification of $\mathbb R^d$.

$\bar{\Omega}_D:=D([0,\infty[,\bar{\mathbb R}^d)$ the set of all right-continuous functions $X:[0,\infty[ \rightarrow 
\bar{\mathbb R}^d$ having the left limits,  such that $X(t)=\infty$, $t>s$, whenever $X(s)=\infty$ or $X(s-)=\infty$.

$\mathcal F_t \equiv \sigma\{X(s) \mid 0 \leq s\leq t, X \in \bar{\Omega}_D\}$ the minimal $\sigma$-algebra containing all cylindrical sets
$\{X \in \bar{\Omega}_D \mid \bigl(X(s_1),\dots,X(s_n)\bigr) \in A, A \subset (\bar{\mathbb R}^{d})^n \text{ is open}\}_{0 \leq s_1 \leq \dots \leq s_n \leq t}$.

$\Omega:=C([0,\infty[,\mathbb R^d)$ denotes the set of all continuous functions $X:[0,\infty[ \rightarrow \mathbb R^d$.

$\mathcal G_t:=\sigma\{X(s) \mid  0 \leq s \leq t, X \in \Omega\}$, $\mathcal G_\infty:=\sigma\{X(s) \mid  0 \leq s < \infty, X \in \Omega\}$.

\smallskip

By the classical result, for a given Feller semigroup $T^t$ on $C_\infty(\mathbb R^d)$, there exist probability measures $\{\mathbb P_x\}_{x \in \mathbb R^d}$ on $\mathcal F_\infty \equiv \sigma\{X(s) \mid 0 \leq s < \infty, X \in \bar{\Omega}_D\}$ such that $(\bar{\Omega}_D, \mathcal F_t, \mathcal F_\infty, \mathbb P_x)$ is a Markov process and
$$
\mathbb E_{\mathbb P_x}[f(X(t))]=T^t f(x), \quad X \in \bar{\Omega}_D, \quad f \in C_\infty, \quad x \in \mathbb R^d.
$$

\begin{theorem}[It\^{o} SDE]
\label{mainthm}
Let $d \geq 3$. Assume that $b \in \mathbf{F}_\delta$, $\nabla_r a_{\cdot \ell} \in \mathbf{F}_{\gamma_{r\ell}}$ and $\nabla a \in \mathbf{F}_{\delta_a}$, with $\gamma:=\sum_{r,\ell=1}^d \gamma_{r\ell}$,  $\delta$, $\delta_a$ satisfying, for some $q>2 \vee (d-2)$, the condition \eqref{cond0}
with $\delta$ replaced by $\delta + \delta_a$.
Let $(\bar{\Omega}_D, \mathcal F_t, \mathcal F_\infty, \mathbb P_x)$ be the Feller process
determined by $T^t=e^{-t\Lambda_{C_\infty}(a,\nabla a + b)}$. 
The following is true for every $x \in \mathbb R^d$:

\smallskip

{\rm(\textit{i})} The trajectories of the process are $\mathbb P_x$ a.s.\,finite and continuous on $0 \leq t <\infty$.

\smallskip

We denote $\mathbb P_x\upharpoonright (\Omega,\mathcal G_\infty)$ again by $\mathbb P_x$.

\smallskip

{\rm(\textit{ii})}  $\mathbb E_{\mathbb P_x}\int_0^t |b(X(s))|ds<\infty$, $X \in \Omega$. 

\smallskip

{\rm(\textit{iii})} For any selection of $f \in C_c^\infty$, $f(y):=y_i$, or $f(y):=y_i y_j$, $1 \leq i,j \leq d$, the process 
$$
M^f(t):=f(X(t)) - f(x) + \int_0^t (-a \cdot \nabla^2 f + b\cdot\nabla f)(X(s))ds, \quad t>0,
$$
is a continuous martingale relative to $(\Omega,\mathcal G_t, \mathbb P_x)$; the latter thus determines a weak solution to the SDE \eqref{sde1} on an extension of $(\Omega,\mathcal G_t, \mathbb P_x)$. 
\end{theorem}

See remark \ref{unique_rem} below concerning the uniqueness.

\begin{theorem}[Stratonovich SDE]
\label{mainthm2}
Let $d \geq 3$. Assume that $b \in \mathbf{F}_\delta$, $\nabla_r \sigma_{\cdot j} \in \mathbf{F}_{\delta_{rj}}$ and $\nabla a \in \mathbf{F}_{\delta_a}$, with $\gamma:=\sum_{r,\ell=1}^d \gamma_{r\ell}$, $\delta$, $\delta_a$, $\delta_c$ satisfying, for some $q>2 \vee (d-2)$, the condition \eqref{cond0} with $\delta$ replaced by $\delta + \delta_a + \delta_c$.
Let $(\bar{\Omega}_D, \mathcal F_t, \mathcal F_\infty, \mathbb P_x)$ be the Feller process
determined by $T^t:=e^{-t\Lambda_{C_\infty}(a,\nabla a - c + b)}$. 
The following is true for every $x \in \mathbb R^d$:

\smallskip

{\rm(\textit{i})} The trajectories of the process are $\mathbb P_x$ a.s.\,finite and continuous on $0 \leq t <\infty$.

\smallskip

We denote $\mathbb P_x\upharpoonright (\Omega,\mathcal G_\infty)$ again by $\mathbb P_x$.

\smallskip

{\rm(\textit{ii})}  $\mathbb E_{\mathbb P_x}\int_0^t |b(X(s))|ds<\infty$, $X \in \Omega$. 

\smallskip

{\rm(\textit{iii})} For any selection of $f \in C_c^\infty$, $f(y):=y_i$, or $f(y):=y_i y_j$, $1 \leq i,j \leq d$, the process 
$$
M^f(t):=f(X(t)) - f(x) + \int_0^t (-a \cdot \nabla^2 f + (b - c) \cdot\nabla f)(X(s))ds, \quad t>0,
$$
is a continuous martingale relative to $(\Omega,\mathcal G_t, \mathbb P_x)$; the latter thus determines a weak solution to \eqref{sde3} on an extension of $(\Omega,\mathcal G_t, \mathbb P_x)$.

\end{theorem}

We fix the following approximation of $\sigma$ by smooth matrices:
$
\sigma_n = I+e^{\epsilon_n\Delta}\big(\eta_n (\sigma-I)\big)
$
($\eta_n$ have been defined earlier).
Then we may assume (cf.\,remark \ref{approx_rem} above) that $a_n:=\sigma_n\sigma_n^t \geq 1$, $b_n$ and $c_n$ defined by \eqref{c} satisfy
$$
\nabla_r (a_n)_{\cdot \ell}\in \mathbf{F}_{\gamma_{r\ell}} \;\;(1 \leq r,\ell \leq d), \quad
\nabla a_n \in \mathbf{F}_{\delta_a}, \quad c_n \in \mathbf{F}_{\delta_c}, \quad \nabla a_n -c_n + b_n \in \mathbf{F}_{\delta_a + \delta_c+ \delta}
$$
with $\lambda \neq \lambda(n)$. If the condition \eqref{cond0} is satisfied with $\delta$ replaced by $\delta_a + \delta_c+ \delta$, then the Feller semigroup $e^{-t\Lambda_{C_\infty}(a,\nabla a - c + b)}$ is well defined, and the properties \eqref{conv_c}, \eqref{consist} and \eqref{SF} hold for $e^{-t\Lambda_{C_\infty}(a,\nabla a - c + b)}$.
Thus, Theorem \ref{mainthm2} is a  consequence of Theorem \ref{mainthm}.

\begin{remark}
\label{unique_rem}
In the assumptions of Theorem \ref{mainthm}, assume also that $\|a-I\|_\infty + \delta<1$.
If $\{\mathbb Q_x\}_{x \in \mathbb R^d}$ is another solution to the martingale problem of {\rm(}\textit{iii}{\rm)} such that
$$
\mathbb Q_x=w{\mbox-}\lim_n \mathbb P_x(\tilde{a}_n,\tilde{b}_n) \quad \text{for every $x \in \mathbb R^d$},
$$
where $\tilde{b}_n$, $\tilde{a}_n$ satisfy $1)$, $2)$ with relative bounds $\tilde{\delta}$, $\tilde{\gamma}_{rk}$, $\tilde{\gamma}_a$ fulfilling \eqref{cond0} with $\delta$ replaced by $\delta + \delta_a$, then $\{\mathbb Q_x\}_{x \in \mathbb R^d}=\{\mathbb P_x\}_{x \in \mathbb R^d}.$ See Appendix A for the proof.

The same remark applies to Theorem \ref{mainthm2} provided that $\|a-I\|_\infty + \delta+\delta_c<1$.
\end{remark}

The proof of Theorem \ref{mainthm} follows the approach in \cite{KiS4}. The latter requires a Feller semigroup, $e^{-t\Lambda_{C_\infty}(a,\nabla a + b)}$, and the estimates of Lemmas A1 and A2 below.

\begin{lemmaB}
\label{lemB}
Assume that the conditions of Theorem \ref{mainthm} are satisfied. 
There exist constants $\mu_0>0$ and $C_i=C_i(\delta,\gamma,\delta_a,q,\mu)$, $i=1,2$, such that, for all $h \in C_c$ and $\mu>\mu_0$, we have:
\begin{equation}
\label{j_2}
\bigl\|(\mu+\Lambda_{C_\infty}(a,\nabla a + b))^{-1}|b_m| h \bigr\|_{\infty} \leq C_1\||b_m|^{\frac{2}{q}}h\|_q,
\end{equation}
\begin{equation}
\label{rem_j3}
\|(\mu+\Lambda_{C_\infty}(a,\nabla a + b))^{-1}|b_m-b_n|h\|_\infty  \leq C_2 \bigl\| |b_m-b_n|^{\frac{2}{q}} h\bigr\|_q.
\end{equation}
\end{lemmaB}

We will also need a weighted variant of Lemma A1. Define  $$\rho(y)\equiv\rho_l (y):=(1+l |y|^2)^{-\nu},   \;\;\; \nu > \frac{d}{2q}+1, \;\; l>0, \;\; y \in \mathbb R^d.$$ Clearly,
\begin{equation}
\label{eta_two_est}
|\nabla \rho| \leq \nu \sqrt{l}\rho, \quad |\Delta \rho| \leq 2\nu (2\nu + d+2 ) l \rho.
\end{equation}

\begin{lemmaA}
\label{lemA}
Assume that the conditions of Theorem \ref{mainthm} are satisfied. 
There exist constants $\mu_0>0$ and $K_1=K_1(\delta,\gamma,\delta_a,q)$ and $K_2=K_2(\delta,\gamma,\delta_a,q,\mu)$ such that, for all $h \in C_c(\mathbb R^d)$, $\mu > \mu_0$ and sufficiently small $l=l(\delta,\gamma,\delta_a,q)>0$, we have:
\begin{equation}
\label{j_1_w}
\tag{$E_1$}
\bigl\|\rho (\mu+\Lambda_{C_\infty}(a_n,\nabla a_n + b_n))^{-1} h \bigr\|_{\infty} \leq K_1\|\rho h\|_q,
\end{equation}
\begin{equation}
\label{j_2_w}
\tag{$E_2$}
\bigl\|\rho (\mu+\Lambda_{C_\infty}(a_n,\nabla a_n + b_n))^{-1}|b_m| h \bigr\|_{\infty} \leq K_2\| |b_m|^\frac{2}{q} \rho h\|_q. 
\end{equation}
\end{lemmaA}

Lemmas A1 and A2 are the new elements of the approach in \cite{KiS4}. Their proofs differs essentially  from the proofs of the analogous results in \cite{KiS4}.

\begin{remark}
\label{proof_rem}
The assumptions on the matrix $a$ in \cite[Theorem 2]{KiS3} are stated in a somewhat different form than in the present paper, but its proof can carried out  without any significant changes in the assumptions 1), 2).
\end{remark}

\section{Proofs of Lemmas A1 and A2}
The proof of Lemma A1 is obtained via a straightforward modification of the proof of Lemma A2. We will attend to it in the end of this section.

\begin{proof}[Proof of Lemma A2] It suffices to prove \eqref{j_1_w}, \eqref{j_2_w} for $(\mu+\Lambda_{q}(a_n,\nabla a_n + b_n))^{-1}$ (cf.\,\eqref{consist}).

Set $A_q^n:=-\nabla \cdot a_n \cdot \nabla$, $D(A_q^n):=W^{2,q}$. Set $\hat{b}_n:=\nabla a_n + b_n$. Then $\hat{b}_n \in \mathbf{F}_{\delta_0}$, $\delta_0:=\delta_a+\delta$.
Put $u_n := (\mu + \Lambda_q(a_n,\hat{b}_n))^{-1} h$, $0 \leq h \in C_c^1$, where $\Lambda_q(a_n,\hat{b}_n)= A_q^n + \hat{b}_n \cdot \nabla$ $(=-a_n \cdot \nabla^2 + b_n \cdot \nabla)$, $D(\Lambda_q(a_n,\hat{b}_n))=W^{2,q}$, $n \geq 1$. Clearly, $ 0 \leq u_n \in W^{3,q}$.

In order to keep our calculations compact we denote $\eta:=\rho^q$. By \eqref{eta_two_est},
\begin{equation}
\label{eta_two_est2}
\tag{$\ast$}
|\nabla \eta| \leq c_1 \sqrt{l}\eta, \quad |\Delta \eta| \leq c_2 l \eta.
\end{equation}
For brevity, 
we omit index $n$ everywhere below: $u \equiv u_n$, $a \equiv a_n$, $\hat{b} \equiv \hat{b}_n$, $A_q \equiv A_q^n$. 
Denote $w:=\nabla u$. 
Set
$$
I_q:=\sum_{r=1}^d\langle (\nabla_r w)^2 |w|^{q-2} \eta \rangle, \quad
J_q:=\langle (\nabla |w|)^2 |w|^{q-2} \eta \rangle,
$$
$$
I^a_q:=\sum_{r=1}^d\langle (\nabla_r w \cdot a \cdot \nabla_r w) |w|^{q-2} \eta \rangle, \quad
J^a_q:=\langle (\nabla |w| \cdot a \cdot \nabla |w|) |w|^{q-2} \eta \rangle.
$$
Set $[F,G]_-:=FG-GF$.

\smallskip

\noindent\textbf{Proof of \eqref{j_1_w}.}~
We will establish a weighted variant of \eqref{reg_star}, then \eqref{j_1_w} will follow by the Sobolev Embedding Theorem.
We multiply the equation $\mu u  + \Lambda_q(a,\hat{b}) u = h $ by $\phi^{}:=- \nabla \cdot (\eta w |w|^{q-2})$ and integrate:
\begin{equation*}
\mu \langle \eta|w|^q \rangle +\langle A_q w^{}, \eta w|w|^{q-2} \rangle + \langle [\nabla,A_q]_-u^{}, \eta w|w|^{q-2}\rangle = \langle -\hat{b} \cdot \nabla u, \phi \rangle  + \langle h, \phi^{} \rangle,
\end{equation*}
\begin{equation*}
\mu \langle \eta|w|^q \rangle + I_q^{a} + (q-2)J^a_q + R^1_q +  \langle [\nabla,A_q]_-u^{}, w|w|^{q-2}\rangle = \langle -\hat{b} \cdot \nabla u, \phi \rangle + \langle  h, \phi \rangle,
\end{equation*}
where $
R^1_q:=\langle a \cdot \nabla |w|, |w|^{q-1} \nabla \eta \rangle
$ (we will get rid of the terms containing $\nabla \eta$, which we denote by $R_q^{~\cdot}$, towards the end of the proof).
Since $a \geq I$, we have $I_q^a \geq I_q$, $J_q^a \geq J_q$. Thus, we arrive at the \textit{principal inequality}
\begin{equation}
\tag{$\bullet$}
\label{principal_id}
\mu \langle |w|^q \rangle + I_q + (q-2)J_q \leq - \langle [\nabla,A_q]_-u^{}, w|w|^{q-2}\rangle + \langle -\hat{b} \cdot \nabla u, \phi \rangle + \langle  h, \phi \rangle - R^1_q.
\end{equation}
We will estimate the RHS of \eqref{principal_id} in terms of $J_q$ and $I_q$. 

First, we estimate $\langle [\nabla,A_q]_- u^{},\eta w|w|^{q-2}\rangle :=\sum_{r=1}^d\langle [\nabla_r,A_q]_- u^{},\eta w_r|w|^{q-2}\rangle.$ From now on, we omit the summation sign in repeated indices.

\begin{claim}
\label{est_comm}
\begin{align*}
|\langle [\nabla_r,A_q]_- u^{},\eta w_r|w|^{q-2}\rangle| & \leq \alpha \gamma\frac{q^2}{4}J_q + \frac{1}{4\alpha}I_q  + (q-2)\biggl[ \beta \gamma \frac{q^2}{4} + \frac{1}{4\beta} \biggr]J_q \\
& + R_q^2 + \big(\alpha + (q-2)\beta\big)R_q^3 + \big(\alpha + (q-2)\beta\big)\lambda\gamma\langle |w|^q \eta \rangle,  \qquad (\alpha, \beta>0)
\end{align*}
where $R_q^2:=\langle (\nabla_ra_{i\ell})w_\ell, w_r |w|^{q-2} \nabla_i \eta\rangle$, $R_q^3:=\frac{q}{2}\langle \nabla |w| ,  |w|^{q-1} \nabla \eta\rangle + \frac{1}{4} \langle |w|^q \frac{(\nabla \eta)^2}{\eta}\rangle$.

\end{claim}
\begin{proof}[Proof of Claim \ref{est_comm}]Note that $ [\nabla,A_q]_- u = -\nabla \cdot \nabla_r a \cdot \nabla$. Thus,
\begin{align*}
\langle [\nabla_r,A_q]_- u^{},\eta w_r|w|^{q-2}\rangle & = \langle (\nabla_r a_{i\ell})w_\ell, \eta(\nabla_i w_r)|w|^{q-2} \rangle \\
& + (q-2)\langle (\nabla_ra_{i\ell})w_\ell,\eta w_r|w|^{q-3}\nabla_i|w| \rangle  + R_q^2.
\end{align*}
By quadratic inequality,
\begin{align*}
|\langle [\nabla_r,A_q]_- u^{},\eta w_r|w|^{q-2}\rangle| & \leq \alpha \big\langle \textstyle{\sum_{r,\ell}}(\nabla_r a_{\cdot \ell})^2 |w|^{q}\eta \big\rangle + \frac{1}{4\alpha}I_q \\
& + (q-2)\biggl[\beta \big\langle \textstyle{\sum_{r,\ell}}(\nabla_r a_{\cdot \ell})^2 |w|^{q}\eta \big\rangle + \frac{1}{4\beta}J_q \biggr] + R_q^2.
\end{align*}
We use $\nabla_r a_{\cdot \ell} \in \mathbf{F}_{\gamma_{r\ell}}$, i.e.
$
\langle (\nabla_r a_{\cdot \ell})^2 \varphi^2 \rangle \leq \gamma_{r\ell} \langle |\nabla \varphi|^2\rangle  + \lambda\gamma\langle |\varphi|^2 \rangle$, $\varphi \in W^{1,2}$, so that
\begin{equation}
\label{a_est}
\big\langle \textstyle{\sum_{r,\ell}}(\nabla_r a_{\cdot \ell})^2 |w|^{q}\eta \big\rangle \leq \gamma\frac{q^2}{4}J_q + R_q^3 + \lambda\gamma\langle |w|^q \eta \rangle,
\end{equation}
where $\gamma=\sum_{r,\ell}\gamma_{r\ell}$. The proof of Claim \ref{est_comm} is completed.
\end{proof}

We estimate the term $\langle -\hat{b}\cdot w, \phi^{} \rangle$ in \eqref{principal_id} as follows.

\begin{claim}
\label{b_est_lem}
There exist constants $C_i$ \rm{($i=0,1,3$)} such that
\begin{align*} 
 \langle - \hat{b} \cdot w, \phi \rangle & \leq \biggl[\big(\sqrt{\delta_0}\sqrt{\delta_a}+\delta_0\big)\frac{q^2}{4}+(q-2)\frac{q\sqrt{\delta_0}}{2}\biggr]J_q \\
&+ \|a-I\|_\infty \biggl[\alpha_1 \delta_0 \frac{q^2}{4}J_q + \frac{1}{4\alpha_1}I_q\biggr]
  + C_0\|w\|^q_q + C_1\|\eta^{\frac{1}{q}}w\|_q^{q-2} \|\eta^{\frac{1}{q}}h\|^2_q + C_2 R_q^3 +  R_q^4, \qquad (\alpha_1>0)
\end{align*}
where $R_q^4:=- \langle \nabla \eta, w |w|^{q-2}(- \hat{b}\cdot w)\rangle \rangle$.
\end{claim}

\begin{proof}[Proof of Claim \ref{b_est_lem}]
We have $\phi = \eta(-\Delta u)|w|^{q-2} - \eta |w|^{q-3} w \cdot \nabla |w| - \nabla \eta \cdot w|w|^{q-2}$, so
\begin{align*}
 \langle-\hat{b}\cdot w, \phi \rangle & =\langle - \Delta u, \eta|w|^{q-2}(-\hat{b}\cdot w)\rangle - (q-2) \langle  w \cdot \nabla |w|,\eta |w|^{q-3}(-\hat{b}\cdot w)\rangle +R_q^4 \\
&=:F_1+F_2 + R_q^4.
\end{align*}
Set $B_q:=\langle \eta \hat{b}^2|w|^{q}\rangle$. We have
$$
F_2
\leq (q-2) B_{q}^\frac{1}{2} J_{q}^\frac{1}{2}.$$
Next, we bound $F_1$. We represent $-\Delta u=\nabla \cdot (a-I)\cdot w -\mu u - \hat{b}\cdot w + h$,
and evaluate: $\nabla \cdot (a-I)\cdot w=\nabla a \cdot w + (a-I)_{i\ell}\nabla_i w_\ell$, so
\begin{align*}
F_1& =\langle \nabla \cdot (a-I)\cdot w, \eta |w|^{q-2}(-\hat{b}\cdot w)\rangle  + \langle (-\mu u - \hat{b}\cdot w + h),\eta |w|^{q-2}(-\hat{b}\cdot w) \rangle \\
& = \langle \nabla a \cdot w, \eta |w|^{q-2}(-\hat{b}\cdot w) \rangle \\
& + \langle (a-I)_{i\ell}\nabla_i w_\ell, \eta |w|^{q-2}(-\hat{b}\cdot w) \rangle \\
& + \langle (-\mu u - \hat{b}\cdot w + h),\eta |w|^{q-2}(-\hat{b}\cdot w) \rangle.
\end{align*}
Set $P_q:=\langle \eta (\nabla a)^2 |w|^{q} \rangle $. We bound $F_1$ from above by applying consecutively the following estimates:

\smallskip

$1^\circ$) $\langle \nabla a \cdot w, \eta |w|^{q-2}(-\hat{b}\cdot w) \rangle \leq P_{q}^{\frac{1}{2}}B_{q}^{\frac{1}{2}}$.

\smallskip

$2^\circ$) $\langle  (a-I)_{i\ell}\nabla_i w_\ell, \eta |w|^{q-2}(-\hat{b}\cdot w) \rangle 
\leq \|a-I\|_\infty I_q^{\frac{1}{2}} B_{q}^{\frac{1}{2}} \leq \|a-I\|_\infty \bigl(\alpha_1 B_q + \frac{1}{4\alpha_1}I_q\bigr)$.

\smallskip

$3^\circ$) $\langle \mu  u , \eta |w|^{q-2} \hat{b} \cdot w  \rangle  \leq \frac{\mu}{\mu-\mu_1} B_{q}^\frac{1}{2} \|\eta^{\frac{1}{q}}w\|_q^\frac{q-2}{2}  \|\eta ^{\frac{1}{q}}h\|_q$ for some $\mu_1>0$, for all $\mu>\mu_1$.

Indeed, $\langle \mu  u , \eta|w|^{q-2} (-\hat{b} \cdot w) \rangle \leq \mu  B_q^\frac{1}{2} \|\eta^{\frac{1}{q}}w\|_q^\frac{q-2}{2} \|\eta^{\frac{1}{q}}u\|_q$ and $\|\eta^{\frac{1}{q}}u\|_q \leqslant (\mu-\mu_1)^{-1}\|\eta^{\frac{1}{q}}h\|_q$, $\mu>\mu_1$, for appropriate $\mu_1>0$. To prove the last estimate, 
we multiply $(\mu+\Lambda_q(a,\hat{b}))u=h$ by $\eta u^{q-1}$ to obtain
$$
\mu\langle u, \eta u^{q-1} \rangle - \langle \nabla \cdot a \cdot w,\eta u^{q-1}\rangle  = \langle - \hat{b} \cdot w,\eta u^{q-1}\rangle + \langle h, \eta u^{q-1}\rangle,
$$
\begin{align*}
\mu\|\eta^{\frac{1}{q}}u\|_q^q  +\frac{4(q-1)}{q^2}\langle \eta \nabla u^{\frac{q}{2}}\cdot a \cdot \nabla u^{\frac{q}{2}}\rangle  + R_q^5 = \langle - \hat{b} \cdot w,\eta u^{q-1}\rangle + \langle h, \eta u^{q-1}\rangle,
\end{align*}
where $R_q^5:=\frac{2}{q}\bigl\langle a \cdot \nabla u^{\frac{q}{2}}, (\nabla \eta) u^{\frac{q}{2}} \bigr\rangle$. In the RHS we apply the quadratic inequality to $\langle - \hat{b} \cdot \nabla u,\eta u^{q-1}\rangle$  to obtain
\begin{align*}
\mu\|\eta^{\frac{1}{q}}u\|_q^q & +\frac{4(q-1)}{q^2}\langle \eta \nabla u^{\frac{q}{2}}\cdot a \cdot \nabla u^{\frac{q}{2}}\rangle  + R_q^5 \\
& \leq \kappa \frac{2}{q}\langle \eta (\nabla u^{\frac{q}{2}})^2\rangle + \frac{1}{2\kappa q}\langle \eta \hat{b}^2 u^{q} \rangle  + \langle h, \eta u^{q-1}\rangle \qquad (\kappa>0),
\end{align*}
\begin{align*}
\mu\|\eta^{\frac{1}{q}}u\|_q^q & +\frac{4(q-1)}{q^2}\langle \eta \nabla u^{\frac{q}{2}}\cdot a \cdot \nabla u^{\frac{q}{2}}\rangle  + R_q^5 \\
& \leq \kappa \frac{2}{q}\langle \eta (\nabla u^{\frac{q}{2}})^2\rangle + \frac{1}{2\kappa q}\langle \eta \hat{b}^2 u^{q} \rangle  + \|\eta^{\frac{1}{q}}h\|_q \|\eta^{\frac{1}{q}}u\|_q^{q-1}.
\end{align*}
Since $a \geq I$, we can replace in the LHS $\langle \eta \nabla u^{\frac{q}{2}}\cdot a \cdot \nabla u^{\frac{q}{2}}\rangle$ by $\langle \eta (\nabla u^{\frac{q}{2}})^2\rangle$. By $\hat{b} \in \mathbf{F}_{\delta_0}$, $\langle \eta \hat{b}^2 u^{q} \rangle \leq \delta_0 \langle \eta (\nabla u^{\frac{q}{2}})^2\rangle + 2 \langle \nabla u^{\frac{q}{2}} , \nabla \eta \rangle  + \langle(\nabla \eta)^2 u^{q} \rangle  + \lambda\delta_0 \langle\eta u^{q} \rangle$, 
and thus we arrive at
$$
\bigl(\mu- \mu_1 \bigr)\|\eta^{\frac{1}{q}}u\|_q^q + \biggl[\frac{4(q-1)}{q^2} - \kappa \frac{2}{q} - \frac{1}{2\kappa q} \delta_0 \biggr]\langle \eta (\nabla u^{\frac{q}{2}})^2\rangle \leq - R_q^5 + R_q^6 + \|\eta^{\frac{1}{q}}h\|_q \|\eta^{\frac{1}{q}}u\|_q^{q-1},
$$  
where $\mu_1:=\lambda \delta_0$, $R_q^6:=\frac{1}{2\kappa q} \bigl(2 \langle \nabla u^{\frac{q}{2}} , \nabla \eta \rangle  + \langle(\nabla \eta)^2 u^{q} \rangle \bigr)$. We select  $\kappa:=\frac{\sqrt{\delta_0}}{2}$. Then, since $q>\frac{2}{2-\sqrt{\delta_0}}$, the coefficient of $\langle \eta (\nabla u^{\frac{q}{2}})^2\rangle $ is positive.
In turn, by \eqref{eta_two_est2}, $$-R_q^5 \leq c_2\sqrt{l} \|a\|_\infty \bigl\langle |\nabla u^{\frac{q}{2}}|, \eta u^{\frac{q}{2}} \bigr\rangle \leq \frac{c_2}{2}\sqrt{l}\|a\|_\infty(\langle \eta (\nabla u^{\frac{q}{2}})^2 \rangle + \langle \eta u^q\rangle ).$$ We estimate $R_q^6$ similarly. The required estimate $(\mu-\mu_1)\|\eta^{\frac{1}{q}}u\|_q \leq \|\eta^{\frac{1}{q}}h\|_q$ now follows upon selecting $l$  sufficiently small in the definition of $\eta$ $(=\rho^q)$ at expense of increasing $\mu_1$ slightly. This completes the proof of $3^\circ$).

\smallskip

$4^\circ$) $\langle \hat{b} \cdot w, \eta|w|^{q-2} \hat{b} \cdot w \rangle = B_{q} .$

\smallskip

$5^\circ$) $\langle h, \eta|w|^{q-2} (- \hat{b} \cdot w)\rangle |\leq B_{q}^\frac{1}{2} \|\eta^{\frac{1}{q}}w\|_q^\frac{q-2}{2} \|\eta^{\frac{1}{q}}h\|_q.$

\smallskip

In $3^\circ$) and $5^\circ$) we estimate $B_{q}^\frac{1}{2} \|\eta^{\frac{1}{q}}w\|_q^\frac{q-2}{2} \|\eta^{\frac{1}{q}}h\|_q \leq \varepsilon_0 B_q + \frac{1}{4\varepsilon_0}\|\eta^{\frac{1}{q}}w\|_q^{q-2} \|\eta^{\frac{1}{q}}h\|^2_q$ ($\varepsilon_0>0$).

\smallskip

The above estimates yield:
\begin{align*}
 \langle-\hat{b}\cdot w, \phi \rangle  & = F_1+F_2+R_q^4 \\
& \leq P_q^{\frac{1}{2}}B_{q}^{\frac{1}{2}} + \|a-I\|_\infty I_{q}^{\frac{1}{2}}B_{q}^{\frac{1}{2}}+B_{q} + (q-2)B_{q}^{\frac{1}{2}}J_{q}^{\frac{1}{2}} \\
& + \varepsilon_0 \bigg(\frac{\mu}{\mu-\mu_1}+1 \bigg) B_q +  C_1(\varepsilon_0)\|\eta^{\frac{1}{q}}w\|_q^{q-2} \|\eta^{\frac{1}{q}}h\|^2_q + R_q^4.
\end{align*}
Selecting $\varepsilon_0>0$ sufficiently small, using that the assumption on $\delta_0$, $\delta_a$ are strict inequalities, we can and will ignore below the terms multiplied by $\varepsilon_0$. 

Finally, we use in the last estimate: By $\hat{b} \in \mathbf{F}_{\delta_0}$,
$$B_q \leq \frac{q^2}{4}\delta_0 J_q + R_q^3 + \lambda\delta_0\langle |w|^q \eta \rangle$$ (cf.\,\eqref{a_est}), and by  $\nabla a \in \mathbf{F}_{\delta_a}$,
\begin{equation*}
P_q \leq  \frac{ q^2}{4} \delta_a J_q + R_q^3 + \lambda\delta_a \|w\|_q^q.  
\end{equation*}
This yields Claim \ref{b_est_lem}.
\end{proof}

We estimate the term $\langle h, \phi \rangle$ in \eqref{principal_id} as follows.

\begin{claim} 
\label{f_est_lem}
For each $\varepsilon_0>0$ there exists a constant $C=C(\varepsilon_0)<\infty$  such that 
$$
\langle h, \phi \rangle \leq \varepsilon_0 I_q + C \|w\|_q^{q-2} \|h\|^2_q + R_q^7,
$$
where $R_q^7:=-\langle \nabla \eta \cdot w|w|^{q-2},h\rangle$.
\end{claim}

\begin{proof}[Proof of Claim \ref{f_est_lem}]
We have:
\[
 \langle h, \phi \rangle =\langle - \Delta u, \eta|w|^{q-2}h\rangle - (q-2) \langle \eta|w|^{q-3} w \cdot \nabla |w|,h\rangle + R_q^7=:F_1+F_2+R_q^7.
 \]
Due to $|\Delta u|^2 \leq d |\nabla_r w|^2$ and $\langle \eta |w|^{q-2}h^2\rangle \leq \|\eta^{\frac{1}{q}}w\|_q^{q-2}\|\eta^{\frac{1}{q}}h\|^2_q$,
\[
F_1 \leq \sqrt{d}I_q^{\frac{1}{2}}\|\eta^{\frac{1}{q}}w\|_q^{\frac{q-2}{2}}\|\eta^{\frac{1}{q}}h\|_q, \qquad F_2 \leq (q-2)J_q^{\frac{1}{2}}\|\eta^{\frac{1}{q}}w\|_q^{\frac{q-2}{2}}\|\eta^{\frac{1}{q}}h\|_q.  
\]
Now the standard quadratic estimates yield Claim \ref{f_est_lem}.
\end{proof}

Since the assumption on $\gamma$, $\delta_0$, $\delta_a$ in the theorem are strict inequalities, we can select $\varepsilon_0>0$ sufficiently small so that we can ignore the term $\varepsilon_0 I_q$ in Claim \ref{f_est_lem}

Applying the estimates of Claims \ref{est_comm}, \ref{b_est_lem} and \ref{f_est_lem} in \eqref{principal_id}, we arrive at: There exists $\mu_0>\mu_1$ such that
\begin{equation*}
\begin{array}{ll}
&(\mu-\mu_0)\|w\|_q^q + I_q + (q-2)J_q - \alpha \gamma \frac{q^2}{4}J_q - \frac{1}{4\alpha}I_q - (q-2)\biggl[\beta\gamma\frac{q^2}{4}+\frac{1}{4\beta} \biggr]J_q \\
&-\biggl((\sqrt{\delta_0}\sqrt{\delta_a} + \delta_0)\frac{q^2}{4} + (q-2)\frac{q\sqrt{\delta_0}}{2} \biggr)J_q - \|a-I\|_\infty \biggl(\alpha_1 \delta_0 \frac{q^2}{4}J_q + \frac{1}{4\alpha_1}I_q \biggr) \\
& \leq C\|\eta^{\frac{1}{q}}h\|^q_q -R_q^1 + R_q^2 + CR_q^3 + R_q^4 + R_q^7.
\end{array}
\end{equation*}
We select $\alpha=\beta:=\frac{1}{q\sqrt{\gamma}}$, $\alpha_1:=\frac{1}{q\sqrt{\delta_0}}$.
By the assumptions of the theorem, the coefficient of $I_q$ 
$$
1-\frac{q}{4}(\sqrt{\gamma} + \|a-I\|_\infty \sqrt{\delta_0}) - \varepsilon_0> 0,
$$
so, by $I_q \geq J_q$, 
\begin{align*}
&(\mu-\mu_0)\|w\|_q^q + \bigg[(q-1)\big(1-\frac{q\sqrt{\gamma}}{2}\big) - (\sqrt{\delta_0}\sqrt{\delta_a} + \delta_0)\frac{q^2}{4} - (q-2)\frac{q\sqrt{\delta_0}}{2} -\|a-I\|_\infty \frac{q\sqrt{\delta_0}}{2} \biggr]J_q \\
& \leq C\|\eta^{\frac{1}{q}}h\|^q_q -R_q^1 + R_q^2 + CR_q^3 + R_q^4 + R_q^7.
\end{align*}
By the assumptions of the theorem the coefficient of $J_q$ is positive.
Selecting $l$ in the definition of $\eta$ sufficiently small,  we eliminate the terms $R_q^i$ ($i=1,2,3,4,7$) using the estimates \eqref{eta_two_est2} as in the proof of $3^\circ$), at expense of increasing $\mu_0$ and decreasing the coefficient of $J_q$ slightly, arriving at
$$
(\mu-\mu_0)\|w\|_q^q + c J_q \leq C\|\eta^{\frac{1}{q}}h\|^q_q, \quad c>0.
$$
In $J_q \equiv \frac{4}{q^2}\langle \eta (\nabla |\nabla u|^{\frac{q}{2}})^2 \rangle$, we commute $\eta$ and $\nabla$ using \eqref{eta_two_est2}, arriving at 
$$
\langle (\nabla |\nabla (\eta^{\frac{1}{q}} u)|^{\frac{q}{2}})^2\rangle \leq C' \|\eta^{\frac{1}{q}}h\|^q_q.
$$
Applying the Sobolev Embedding Theorem twice, we obtain \eqref{j_1_w}.

\medskip

\noindent\textbf{Proof of \eqref{j_2_w}.}
We modify the proof of \eqref{j_1_w}. Now, $u = (\mu + \Lambda_q(a,\hat{b}))^{-1} |b_m|h$, where $0 \leqslant h \in C_c$. The modification amounts to
 replacing  $h$ by $|b_m|h$ which requires the following changes in the estimates involving $h$.
Namely, in the proof of Claim \ref{b_est_lem}, we replace $3^\circ$) with 

$3'$) $\langle \hat{b} \cdot w, \eta|w|^{q-2} \mu u_n \rangle 
\leqslant \mu C(\mu) B_q^\frac{1}{2} \|\eta^{\frac{1}{q}}w\|_q^\frac{q-2}{2} \|\eta^{\frac{1}{q}} |b_m|^{\frac{2}{q}}h\|_q
$
where we used $
\|\eta^\frac{1}{q}u_n\|_q \leqslant C(\mu)\|\eta^{\frac{1}{q}} |b_m|^{\frac{2}{q}}h\|_q$. The proof of the last estimate follows the proof in $3^\circ$), but now we estimate $\langle h, \eta u^{q-1}\rangle$ by Young's inequality:
\begin{align*}
\langle |b_m|h,\eta u^{q-1}\rangle & \leq \frac{q-1}{q}\sigma^\frac{q}{q-1}\langle \eta |b_m|^{\frac{q-2}{q-1}} u^p\rangle + \frac{\sigma^{-q}}{q}\langle \eta |b_m|^2 h^q\rangle \qquad (\sigma>0)\\
& \leq \frac{q-1}{q}\sigma^\frac{q}{q-1}\langle \eta (1+|b_m|^2)u^q\rangle + \frac{\sigma^{-q}}{q}\langle \eta |b_m|^2 h^q\rangle.
\end{align*}
It remains to apply $b_m \in \mathbf{F}_\delta$ with $\lambda \neq \lambda(m)$ in order to estimate $\langle \eta(1+|b_m|^2)u^q\rangle$ in terms of $\langle \eta (\nabla u^{\frac{q}{2}})^2 \rangle$, $\|\eta^{\frac{1}{q}}u\|_q^q$ and the terms containing $\nabla \eta$ which can be discarded at expense on increasing $\mu_0$. We select $\sigma>0$ sufficiently small to obtain the required estimate.

\smallskip

We replace 
$5^\circ$) by

$5')$ $\langle |b_m| h, \eta|w|^{q-2} (- \hat{b} \cdot w)\rangle |\leq B_q^{\frac{1}{2}} \big\langle \eta (|b_m| h)^2 |w|^{q-2} \big\rangle^{\frac{1}{2}},$
where, in turn,
\begin{align}
\big\langle \eta (|b_m| h)^2 |w|^{q-2} \big\rangle 
& \leq \frac{q-2}{q}\epsilon^{\frac{q}{q-2}}  \bigl\langle \eta |b_m|^2 |w|^{q} \bigr\rangle
+ \frac{2}{q}\epsilon^{-\frac{2}{q}} \bigl\langle \eta |b_m|^2 h^q \bigr\rangle \notag \\
& \text{(use $b_m \in \mathbf{F}_\delta$ with $\lambda \neq \lambda(m)$)} \\
& \leq \frac{q-2}{q}\epsilon^{\frac{q}{q-2}}  \biggl[\frac{q^2}{4}\delta J_q + R_q^3 + \lambda\delta \langle \rangle|w|^q\eta \biggr]
+ \frac{2}{q}\epsilon^{-\frac{2}{q}} \bigl\langle \eta |b_m|^2 h^q \bigr\rangle
\label{extra_est}
\end{align}
where $\epsilon>0$ is to be chosen sufficiently small.

In the proof of Claim \ref{f_est_lem}, we replace the estimate $\langle \eta |w|^{q-2}h^2\rangle \leq \|\eta^{\frac{1}{q}}w\|_q^{q-2}\|\eta^{\frac{1}{q}}h\|^2_q$ by \eqref{extra_est}. 
The analogue of $R_q^7$ is $-\langle \nabla \eta \cdot w |w|^{q-2}, |b_m|h \rangle$, which we eliminate by estimating using \eqref{eta_two_est2}
$$-\langle \nabla \eta \cdot w |w|^{q-2}, |b_m|h \rangle \leqslant c_1^2 l \bigl\langle \eta (|b_m|h)^2|w|^{q-2}\bigr\rangle^\frac{1}{2}  \|\eta^{\frac{1}{q}}w\|_q^\frac{q}{2},$$
applying \eqref{extra_est} to the first term in the RHS, and selecting $l$ in the definition of $\eta$ sufficiently small.

The rest of the proof repeats the proof of \eqref{j_1_w}.
\end{proof}

\begin{proof}[Proof of Lemma A1]
The proof of \eqref{j_2} repeats the proof of \eqref{j_2_w} with $\rho$ taken to be $\equiv 1$. The proof of \eqref{rem_j3} also repeats the proof of \eqref{j_2_w} with $\rho \equiv 1$ where we take into account that $b_m-b_n \in \mathbf{F}_{\delta}$ with $\lambda \neq \lambda(m,n)$.
\end{proof}

\section{Proof of Theorem \ref{mainthm}}
We follow the approach of \cite{KiS4}.  For the sake of completeness, we have included all the details.

\begin{lemma}
\label{ae_rem}
For every $x \in \mathbb R^d$ and $t>0$, $b_n(X(t)) \rightarrow b(X(t))$, $a_n(X(t)) \rightarrow a(X(t))$ $\mathbb P_x$ a.s. as $n \uparrow \infty$.
\end{lemma}

\begin{proof}[Proof of Lemma \ref{ae_rem}]
The proof repeats the proof of \cite[Lemma 1]{KiS4}.
By \eqref{SF} and the Dominated Convergence Theorem, for any $\mathcal L^d$-measure zero set $G \subset \mathbb R^d$ and every $t>0$, $\mathbb P_x[X(t) \in G]=0$. 
Since $b_n \rightarrow b$, $a_n \rightarrow a$ pointwise in $\mathbb R^d$ outside of an $\mathcal L^d$-measure zero set, we have the required.
\end{proof}

\begin{lemma}
\label{finite_prop} 
For every $x \in \mathbb R^d$ and $t>0$, $\mathbb P_x[X(t)=\infty]=0$.
\end{lemma}
\begin{proof}[Proof of Lemma \ref{finite_prop}]
The proof repeats the proof of \cite[Lemma 2]{KiS4}.
First, let us show that for every $\mu>\mu_0$, 
\begin{equation}
\label{conv_n9}
\int_0^\infty e^{-\mu t}\mathbb E^n_x[\xi_k(X(t))]dt \rightarrow \frac{1}{\mu} \quad \text{ as $k \uparrow \infty$ \textit{uniformly in} $n$}.
\end{equation}
(See \eqref{xi_k} for the definition of $\xi_k$.)
Since $\int_0^\infty e^{-\mu t}\mathbb E^n_x[\mathbf{1}_{\mathbb R^d}(X(t))]dt=\frac{1}{\mu}$, \eqref{conv_n9} is equivalent to
$\int_0^\infty e^{-\mu t}\mathbb E^n_x[(\mathbf{1}_{\mathbb R^d}-\xi_k)(X(t))]dt  \rightarrow 0$ as $k \uparrow \infty$ uniformly in $n$. 
We have
\begin{align*}
& \int_0^\infty e^{-\mu t}\mathbb E^n_x[(\mathbf{1}_{\mathbb R^d}-\xi_k)(X(t))]dt \\
& \text{(we use the Dominated Convergence Theorem)} \\
& = \lim_{r \uparrow \infty}\int_0^\infty e^{-\mu t}\mathbb E^n_x[\xi_r(1-\xi_k)(X(t))]dt \\
& = \lim_{r \uparrow \infty}(\mu+\Lambda_{C_\infty}(a_n,\nabla a_n + b_n))^{-1}[\xi_r(1-\xi_k)](x) \\
& \text{(we apply crucially \eqref{j_1_w})} \\
& \leq \rho(x)^{-1} K_1\lim_{r \uparrow \infty} \|\rho\xi_r (1-\xi_k)\|_p \leq \rho(x)^{-1} K_1\|\rho (1-\xi_k) \|_p \rightarrow 0 \quad \text{as $k \uparrow \infty$},
\end{align*}
which yields \eqref{conv_n9}.

Now, since  $\mathbb E_x[\xi_k(X(t))]=\lim_{n} \mathbb E^n_x[\xi_k(X(t))]$ uniformly on every compact interval of $t \geq 0$, see \eqref{conv_c}, it follows from \eqref{conv_n9} that $$\int_0^\infty e^{-\mu t}\mathbb E_x[\xi_k(X(t))]dt \rightarrow \frac{1}{\mu} \quad \text{ as }k \uparrow \infty.$$

Finally, suppose that  $\mathbb P_x[X(t)=\infty]$ is strictly positive for some $t>0$.
By the construction of $\mathbb P_x$, $t \mapsto \mathbb P_x[X(t)=\infty]$ is non-decreasing, and so
$\varkappa:=\int_0^\infty e^{-\mu t}\mathbb E_x[\mathbf{1}_{X(t)=\infty}]dt>0$. Now,
$$
\frac{1}{\mu}=\int_0^\infty e^{-\mu t}\mathbb E_x[\mathbf{1}_{\bar{\mathbb R}^d}(X(t))]dt \geq \varkappa + \int_0^\infty e^{-\mu t}\mathbb E_x[\xi_k(X(t))]dt.
$$
Selecting $k$ sufficiently large, we arrive at contradiction. 
\end{proof}

Let $\mathbb P^n_x$  be the probability measures associated with $e^{-t\Lambda_{C_\infty}(a_n,\nabla a_n + b_n)}$, $n=1,2,\dots$

Set $\mathbb E_x:=\mathbb E_{\mathbb P_x}$, and $\mathbb E^n_x:=\mathbb E_{\mathbb P_x^n}$.

The space $\Omega_D:=D([0,\infty[,\mathbb R^d)$ is defined to be the subspace of $\bar{\Omega}_D$ $(:=D([0,\infty[,\bar{\mathbb R}^d))$ consisting of the trajectories $X(t) \neq \infty$, $0 \leq t <\infty$. 
Let $\mathcal F'_{t}:=\sigma(X(s) \mid 0 \leq s \leq t, X \in \Omega_D)$, $\mathcal F'_{\infty}:=\sigma(X(s) \mid 0 \leq s <\infty, X \in \Omega_D)$.

By Lemma \ref{finite_prop}, $(\Omega_D,\mathcal  F_\infty')$ has full $\mathbb P_x$-measure in $(\bar{\Omega}_D,\mathcal  F_\infty)$.
We denote the restriction of $\mathbb P_x$ from $(\bar{\Omega}_D,\mathcal  F_\infty)$ to $(\Omega_D,\mathcal  F_\infty')$ again by $\mathbb P_x$.

\begin{lemma}
\label{Y_prop}
For every $x \in \mathbb R^d$ and $g \in C_c^\infty(\mathbb R^d)$, 
$$
g(X(t)) - g(x) + \int_0^t (-a \cdot \nabla^2 g + b\cdot\nabla g)(X(s))ds, $$
is a martingale relative to $(\Omega_D,\mathcal F'_t, \mathbb P_x)$.

\end{lemma}
\begin{proof}We modify the proof of \cite[Lemma 3]{KiS4}.
Fix $\mu>\mu_0$. In what follows, $0<t\leq T<\infty$.

\smallskip

$(\mathbf a)$ $\mathbb E_x \int_0^t \bigl|b\cdot\nabla g \bigr|(X(s))ds<\infty$. Indeed, 
\begin{align*}
& \notag \mathbb E_x \int_0^t \bigl|b\cdot\nabla g \bigr|(X(s))ds  \\
& \notag (\text{we apply Fatou's Lemma, cf.\,Lemma \ref{ae_rem}}) \\
& \notag \leq \liminf_n \mathbb E_x \int_0^t \bigl|b_n\cdot\nabla g \bigr|(X(s))ds  = \liminf_n  \int_0^t e^{-s\Lambda_{C_\infty}(a,\nabla a + b)}\bigl|b_n\cdot\nabla g \bigr|(x)ds   \\
& =\liminf_n  \int_0^t e^{\mu s} e^{-\mu s }e^{-s\Lambda_{C_\infty}(a,\nabla a + b)}\bigl|b_n\cdot\nabla g \bigr|(x)ds  \\
& \notag \leq e^{\mu T} \liminf_n (\mu+\Lambda_{C_\infty}(a,\nabla a + b))^{-1}|b_n| |\nabla g|(x) \\
& \notag \text{(we apply \eqref{j_2} with } h=|\nabla g|) \\
& \notag \leq C_1 e^{\mu T}\liminf_n  \langle |b_n|^2 |\nabla g|^p \rangle^{\frac{1}{p}} \leq C_1 e^{\mu T} 2^\frac{1}{p}\big(\langle |b|^2 |\nabla g|^p \rangle^\frac{1}{p}+ \lim_n \langle |b-b_n|^2 |\nabla g|^p \rangle ^\frac{1}{p}\big) \\
& = C_1 e^{\mu T}2^\frac{1}{p} \langle |b|^2 |\nabla g|^p \rangle^\frac{1}{p} <\infty.
\end{align*}

$(\mathbf a')$ $\mathbb E_x \int_0^t \bigl|a\cdot\nabla^2 g \bigr|(X(s))ds<\infty$ since $a$ is bounded.

\smallskip

$(\mathbf b)$ We have
$$
\mathbb E^n_x[g(X(t))] \rightarrow \mathbb E_x[g(X(t))], 
$$
$$
\mathbb E_x^n \int_0^t ( b_n\cdot\nabla g)(X(s))ds \rightarrow \mathbb E_x\int_0^t (b\cdot\nabla g)(X(s))ds, 
$$ $$\mathbb E_x^n \int_0^t ( a_n\cdot\nabla^2 g)(X(s))ds \rightarrow \mathbb E_x\int_0^t (a\cdot\nabla^2 g)(X(s))ds, 
$$
and also, for $h \in C_c^\infty$,
$$
\mathbb E^n_x\int_0^t
(|b_n| h)(X(s))ds \rightarrow \mathbb E_x\int_0^t
(|b|h)(X(s))ds
$$
as $n \uparrow \infty$.
Indeed, the first  convergence follows from \eqref{conv_c}. The second convergence follows from $(\mathbf c)$ below. The third convergence follows from a straightforward modification $(\mathbf c)$ (use \eqref{SF} and the obvious fact that $a_n \cdot \nabla^2g \rightarrow a \cdot \nabla^2 g$ in $L^p$). The fourth convergence follows from $\mathbb E_x \int_0^t (|b||h|)(X(s))ds<\infty$, a straightforward modification of $(\mathbf a)$.

\smallskip

$(\mathbf c)$ $\mathbb E_x\int_0^t (b_n\cdot\nabla g)(X(s))ds - \mathbb E^n_x\int_0^t (b_n\cdot\nabla g)(X(s))ds\rightarrow 0.$ We have:
\begin{align*}
& \mathbb E_x\int_0^t (b_n\cdot\nabla g)(X(s))ds - \mathbb E^n_x\int_0^t (b_n\cdot\nabla g)(X(s))ds \\
& = \int_0^t \left( e^{-s\Lambda_{C_\infty}(a,\nabla a + b)} - e^{-s\Lambda_{C_\infty}(a_n,\nabla a_n + b_n)} \right) (b_n\cdot\nabla g)(x)ds \\
& =\int_0^t \left( e^{-s\Lambda_{C_\infty}(a,\nabla a + b)} - e^{-s\Lambda_{C_\infty}(a_n,\nabla a_n + b_n)} \right) ((b_n-b_m)\cdot\nabla g)(x)ds  \\
& + \int_0^t \left( e^{-s\Lambda_{C_\infty}(a,\nabla a + b)} - e^{-s\Lambda_{C_\infty}(a_n,\nabla a_n + b_n)} \right) (b_m\cdot\nabla g)(x)ds =: S_1 + S_2,
\end{align*}
where $m$ is to be chosen. 
Arguing as in the proof of $(\mathbf a)$, we obtain:
\begin{align*}
S_1(x) & \leq   
e^{\mu T} (\mu+\Lambda_{C_\infty}(a,\nabla a + b))^{-1}|(b_n-b_m)\cdot\nabla g|(x)  + e^{\mu T} (\mu+\Lambda_{C_\infty}(a_n,\nabla a_n + b_n))^{-1}|(b_n-b_m)\cdot\nabla g|(x).
\end{align*}
Since $b_n-b_m \rightarrow 0$ in $L^2_{\loc}$ as $n,m \uparrow \infty$, \eqref{rem_j3} yields $S_1 \rightarrow 0$ as $n,m \uparrow \infty$.
Now, fix a sufficiently large $m$. Since $e^{-s\Lambda_{C_\infty}(a,\nabla a + b)}=s\text{-}C_\infty\text{-}\lim_n e^{-s\Lambda_{C_\infty}(a_n,\nabla a_n + b_n)}$ uniformly in $0 \leq s \leq T$, cf.~\eqref{conv_c}, 
we have $S_2 \rightarrow 0$ as $n \uparrow \infty$. 
The proof of $(\mathbf c)$ is completed.

\medskip

Now we are in position to complete the proof of Lemma \ref{Y_prop}.
Since $a_n \in [C_c^\infty]^{d \times d}$, $b_n \in [C_c^\infty]^d$, 
$$g(X(t)) - g(x) + \int_0^t (-a_n \cdot \nabla^2 g + b_n\cdot\nabla g)(X(s))ds \text{ is a martingale under $\mathbb P^n_x$,}
$$
so the function
$$
x \mapsto \mathbb E^n_x[g(X(t))] - g(x) +\mathbb E^n_x\int_0^t (-a_n \cdot \nabla^2  g + b_n\cdot\nabla g)(X(s))ds \quad \text{ is identically zero in } \mathbb R^d.
$$
Thus by $(\mathbf b)$, the function
$$
x \mapsto \mathbb E_x[g(X(t))] - g(x) +\mathbb E_x\int_0^t (-a \cdot \nabla^2  g + b\cdot\nabla g)(X(s))ds \quad  \text{ is identically zero in } \mathbb R^d,
$$
i.e. $g(X(t)) - g(x) +\int_0^t (-a \cdot \nabla^2  g + b\cdot\nabla g)(X(s))ds$ is a martingale under $\mathbb P_x$.
\end{proof}

\begin{lemma}
\label{cont_prop}
For $x \in \mathbb R^d$, 
$\Omega$ has full $\mathbb P_x$-measure in $\Omega_D$.

\end{lemma}

\begin{proof}[Proof of Lemma \ref{cont_prop}] The proof repeats the proof of \cite[Lemma 4]{KiS4}. Let $A$, $B$ be arbitrarily bounded closed sets in $\mathbb R^d$, $\dist(A,B)>0$. 
Fix $g \in C_c^\infty(\mathbb R^d)$ such that $g = 0$ on $A$, $g = 1$ on $B$. Set ($X \in \Omega_D$)
$$
M^g(t):=g(X(t)) - g(x) + \int_0^t (-a \cdot \nabla^2 g + b\cdot\nabla g)(X(s))ds, \quad
K^g(t):=\int_0^t \mathbf{1}_A(X(s-))dM^g(s),
$$
then
\begin{align*}
K^g(t) &=\sum_{s \leq t} \mathbf{1}_A \left(X(s-)\right)g(X(s)) +
\int_0^t \mathbf{1}_A(X(s-))\bigl(-a \cdot \nabla^2 g + b\cdot\nabla g \bigr)(X(s))ds \\
& =\sum_{s \leq t} \mathbf{1}_A \left(X(s-)\right)g(X(s)).
\end{align*}
By Lemma \ref{Y_prop}, $M^g(t)$ is a martingale, and hence so is $K^g(t)$. Thus, $\mathbb{E}_x\bigl[\sum_{s \leq t} \mathbf{1}_A (X(s-))g(X(s))\bigr]=0.$ Using the Dominated Convergence Theorem, we obtain
$\mathbb{E}_x\bigl[\sum_{s \leq t} \mathbf{1}_A (X(s-))\mathbf{1}_B(X(s))\bigr]=0$. The proof of Lemma \ref{cont_prop} is completed.
\end{proof}

We denote the restriction of $\mathbb P_x$ from $(\Omega_D, \mathcal F_\infty')$  to $(\Omega,\mathcal G_\infty)$ again by $\mathbb P_x$. Lemma \ref{Y_prop} and Lemma \ref{cont_prop} combined yield

\begin{lemma}
\label{thm1}
For every $x \in \mathbb R^d$ and $g \in C_c^\infty(\mathbb R^d)$, 
$$
g(X(t)) - g(x) + \int_0^t (-a \cdot \nabla^2 g + b\cdot\nabla g)(X(s))ds, \quad X \in \Omega,$$
is a continuous martingale relative to $(\Omega,\mathcal G_t, \mathbb P_x)$.
\end{lemma}

\begin{lemma}
\label{thm2}
For every $x \in \mathbb R^d$ and  $t>0$,
$
\mathbb E_x\int_0^t |b(X(s))|ds<\infty,
$ and, for $f(y)=y_i$ or $f(y)=y_iy_j$, $1 \leq i,j \leq d$,
$$
f(X(t)) - f(x) + \int_0^t (-\Delta f + b\cdot\nabla f)(X(s))ds, \quad X \in C([0,\infty[,\mathbb R^d),
$$
is a continuous martingale relative to $(\Omega,\mathcal G_t, \mathbb P_x)$.
\end{lemma}

\begin{proof}We modify the proof of \cite[Lemma 5]{KiS4}.

Fix a $\upsilon \in C^\infty([0,\infty[)$, $\upsilon (s)=1$ if $0 \leq s \leq 1$, $\upsilon (s)=0$ if $s \geq 2$.
Set
\begin{equation}
\label{xi_k}
\xi_k(y):=\left\{
\begin{array}{ll}
\upsilon (|y|+1-k) & |y| \geq k, \\
1 & |y|<k.
\end{array}
\right.
\end{equation}
Define $f_k:=\xi_k f \in C_c^\infty(\mathbb R^d)$.
Set $\alpha:=\|\nabla \xi_k\|_\infty$, $\beta:=\|\Delta \xi_k\|_\infty$ ($\alpha, \beta$ don't depend on $k$).
Fix $0<T<\infty$. In what follows, $0 < t \leq T$.

$(\mathbf a) \quad \mathbb E_x\int_0^t (|b| (|\nabla f| + \alpha|f|)) (X(s))ds <\infty.$

 Indeed, set $\varphi:=|\nabla f| + \alpha|f| \in C \cap W^{1,2}_{\loc}$, $\varphi_k:=\xi_{k+1} \varphi \in C_c \cap W^{1,2}$.
First, let us prove that
\[
\mathbb E^n_x\int_0^t (|b_n| \varphi_k) (X(s))ds  \leq \const \text{ independent of } n, k.
\]
Fix $p > 2 \vee (d-2)$ satisfying \eqref{cond0}. By \eqref{eta_two_est}, $\sqrt{(\rho \varphi)^p} \in W^{1,2}$.
We have
\begin{align*}
& \mathbb E^n_x\int_0^t (|b_n| \varphi_k) (X(s))ds=\int_0^t e^{-s\Lambda_{C_\infty}(a_n,\nabla a_n + b_n)}|b_n| \varphi_k(x)ds \\
& \leq e^{\mu T} (\mu+\Lambda_{C_\infty}(a_n,\nabla a_n + b_n))^{-1}|b_n| \varphi_k(x) \\
& (\text{we apply \eqref{j_2_w}} ) \\
& \leq e^{\mu T} \rho(x)^{-1}K_2 \langle |b_n|^2 (\rho \varphi_k)^p \rangle^{\frac{1}{p}}   \leq e^{\mu T} \rho(x)^{-1}K_2\langle |b_n|^2 (\rho \varphi)^p \rangle^{\frac{1}{p}}  \\
& \big(\text{we use } b_n \in \mathbf F_{\delta}, \lambda \neq \lambda(n)\big) \\
& \leq e^{\mu T} \rho(x)^{-1}K_2 \delta^{\frac{1}{p}} 
\|(\lambda-\Delta)^\frac{1}{2} \sqrt{(\rho \varphi)^p} \|_2^{\frac{2}{p}}
 <\infty.
\end{align*}
By step $(\mathbf b)$ in the proof of Lemma \ref{Y_prop}, $
\mathbb E^n_x\int_0^t
(|b_n| \varphi_k)(X(s))ds \rightarrow \mathbb E_x\int_0^t
(|b|\varphi_k)(X(s))ds$ as $n \uparrow \infty$.
Therefore, $\mathbb E^n_x\int_0^t (|b_n| \varphi_k) (X(s))ds  \leq C$ implies
$
\mathbb E_x\int_0^t (|b| \varphi_k) (X(s))ds  \leq C \; (C \neq C(k)).
$
Now, Fatou's Lemma yields the required.

\smallskip

$(\mathbf b)$ \quad  For every $t>0$, \; 
$
\mathbb E_x\int_0^t (|a \cdot \nabla^2 f| + 2 \alpha |\nabla f| + \beta |f|)(X(t)) ds <\infty.
$

The proof is similar to the proof of $(\mathbf a)$ (use \eqref{j_1_w} instead of \eqref{j_2_w}). 

\smallskip

$(\mathbf c)$ \quad For every $t>0$,\;
$
\mathbb E_x[|f|(X(t))]<\infty. 
$

Indeed, set $g(y):=1+|y|^2$, $y \in \mathbb R^d$. Since $|f| \leq g$, it suffices to show that $\mathbb E_x[g(X(t))]<\infty$. Set $g_k(y):=\xi_k(y) g(y)$. By Lemma \ref{thm1},
$$
\mathbb E_x[g_k(X(t))]  = g_k(x) - \mathbb E_x \int_0^t  (-a \cdot \nabla^2 g_k)(X(s))ds -  \mathbb E_x \int_0^t (b \cdot \nabla g_k)(X(s))ds. 
$$
Note that
$$
\sup_k\mathbb E_x\int_0^t (|b| |g_k|) (X(s))ds <\infty, \quad \sup_k\mathbb E_x\int_0^t |a \cdot \nabla^2 g_k|(X(s)) ds <\infty
$$
for, arguing as in the proofs of $(\mathbf a)$ and $(\mathbf b)$, we have:
$$
\mathbb E_x\int_0^t (|b| (|\nabla g| + \alpha|g|)) (X(s))ds <\infty, \quad \mathbb E_x\int_0^t (|a \cdot \nabla^2 g| + 2 \alpha |\nabla g| + \beta |g|)(X(t)) ds <\infty.
$$
Therefore, $\sup_k \mathbb E_x[g_k(X(t))]<\infty$, and so, by the Monotone Convergence Theorem, $\mathbb E_x[g(X(t))]<\infty$. This completes the proof of $(\mathbf{c})$.

\smallskip

Let us complete the proof of Lemma \ref{thm2}. 
By $(\mathbf a)$, $
\mathbb E_x\int_0^t |b(X(s))|ds<\infty$.
By $(\mathbf a)$-$(\mathbf c)$,
$$
M^f(t):=f(X(t)) - f(x) + \int_0^t (-a \cdot \nabla f + b\cdot\nabla f)(X(s))ds, \quad t>0,
$$
satisfies $\mathbb E_x[|M^{f}(t)|]<\infty$ for all $t>0$.
By Lemma \ref{thm1}, for every $k$, $M^{f_k}(t)$ is a martingale relative to $(\Omega,\mathcal G_t, \mathbb P_x)$.  By $(\mathbf a)$ and the Dominated Convergence Theorem, since $|\nabla f_k| \leq |\nabla f| + \alpha|f|$ for all $k$, we have
$
\mathbb E_x\int_0^t (b \cdot \nabla f_k) (X(s))ds \rightarrow \mathbb E_x\int_0^t (b \cdot \nabla f) (X(s))ds.
$
By $(\mathbf b)$,
$
\mathbb E_x\int_0^t (a \cdot \nabla^2 f_k) (X(s))ds \rightarrow \mathbb E_x\int_0^t (a \cdot \nabla^2 f) (X(s))ds.
$
By $(\mathbf c)$,
$
\mathbb E_x[f_k(X(t))] \rightarrow \mathbb E_x[f(X(t))]. 
$
So,
$M^{f}(t)$ is also a martingale on $(\Omega,\mathcal G_t, \mathbb P_x)$.
The proof of Lemma \ref{thm2} is completed.
\end{proof}

We are in position to complete the proof of Theorem \ref{mainthm}(\textit{i})-(\textit{iii}). Lemma \ref{cont_prop} yields (\textit{i}). Lemma \ref{thm2} yields (\textit{ii}) and (\textit{iii}). 
The proof of Theorem \ref{mainthm} is completed.

\appendix

\section{}

We prove the assertion of remark \ref{unique_rem}.
For $f \in C_c^\infty$, $x \in \mathbb R^d$, denote
$$
R^n_{\mu}f(x):=\mathbb E_{\mathbb P^n_x}\int_0^\infty e^{-\mu s} f(X(s))ds\quad \biggl(=(\mu+\Lambda_{C_\infty}(\tilde{a}_n,\nabla \tilde{a}_n + \tilde{b}_n))^{-1}f(x)\biggr), $$
$$R^Q_{\mu}f(x):=\mathbb E_{\mathbb Q_x}\int_0^\infty e^{-\mu s} f(X(s))ds, \quad \mu>0.
$$
Let us show that $(\mu+\Lambda_{C_\infty}(a,\nabla a + b))^{-1}f(x)=R^Q_\mu f(x)$ for all $\mu>0$ sufficiently large; this would imply that $\{\mathbb Q_x\}_{x \in \mathbb R^d}=\{\mathbb P_x\}_{x \in \mathbb R^d}$.

\smallskip

We have:

\smallskip

1) $
R^n_\mu f(x) \rightarrow R_\mu^Q f(x)$ (the assumption).

\medskip

2) $\|R_\mu^Qf\|_2 \leqslant (\mu-\omega_2)^{-1}\|f\|_2$,  $\mu>\omega_2$.

\smallskip

Indeed,
$R_\mu^n f=(\mu+\Lambda_{2}(\tilde{a}_n,\nabla \tilde{a}_n + \tilde{b}_n))^{-1}f$,  $f \in C_c^\infty$. Since $e^{-t\Lambda_{2}(\tilde{a}_n,\nabla \tilde{a}_n + \tilde{b}_n)}$ is a quasi contraction on $L^2$, $\|(\mu+\Lambda_{2}(\tilde{a}_n,\nabla \tilde{a}_n + \tilde{b}_n))^{-1}\|_{2 \rightarrow 2} \leqslant (\mu-\omega_2)^{-1}$, $\mu>\omega_2$, $0<\omega_2 \neq \omega_2(n)$. Thus, $\|R_\mu^nf\|_2 \leqslant (\mu-\omega_2)^{-1}\|f\|_2$ for all $n$.
Now 2) follows from 1) by a weak compactness argument in $L^2$.

By 2), $R_\mu^Q$  admits extension by continuity to $L^2$, which we denote by $R_{\mu,2}^Q$.

\smallskip

3) $\|(-(a-I)\cdot \nabla^2 + b\cdot\nabla)(\mu-\Delta)^{-1}\|_{2\rightarrow 2} \leqslant \|a-I\|_\infty + \delta$ (we use $b \in \mathbf{F}_\delta$).

\smallskip

4) $(\mu+\Lambda_{2}(a,\nabla a + b))^{-1}f = (\mu-\Delta)^{-1}\big(1+((a-I)\cdot \nabla^2 - b\cdot\nabla) (\mu-\Delta)^{-1}\big)^{-1} f$.

Indeed, by our assumptions $\|a-I\|_\infty + \delta<1$, so in view of 3) the RHS is well defined. 
Clearly, 4) holds for $a=a_n$, $b=b_n$. We pass to the limit $n \rightarrow \infty$ using \eqref{conv1}.

\smallskip

5) $(\mu+\Lambda_{C_\infty}(a,\nabla a + b))^{-1} f = R_\mu^Q f$ a.e.~on  $\mathbb R^d$.

Indeed, since $\{\mathbb Q_x\}$ is a weak solution of \eqref{sde1}, we have by It\^{o}'s formula 
$$
(\mu-\Delta)^{-1} h= R_\mu^Q[\big(1+((a-I)\cdot \nabla^2 - b\cdot\nabla) (\mu-\Delta)^{-1}\big)h], \quad h \in C_c^\infty.
$$
Since $\|\big(1+((a-I)\cdot \nabla^2 - b\cdot\nabla) (\mu-\Delta)^{-1}\big)\|_{2 \rightarrow 2}<\infty$ (by 3)), we have, in view of 2),
$$
(\mu-\Delta)^{-1} g= R_{\mu,2}^Q[\big(1+((a-I)\cdot \nabla^2 - b\cdot\nabla) (\mu-\Delta)^{-1}\big)g], \quad g \in L^2.
$$
Take $g=\big(1+((a-I)\cdot \nabla^2 - b\cdot\nabla) (\mu-\Delta)^{-1}\big)^{-1}f$, $f \in C_c^\infty$. Then by 4)
$
(\mu+\Lambda_2(b))^{-1}f=R_{\mu,2}^Q f
$. By the consistency property $(\mu+\Lambda_{C_\infty}(b))^{-1}|_{C_c^\infty \cap L^2}=(\mu+\Lambda_2(b))^{-1}|_{C_c^\infty \cap L^2}$, and the result follows.

\smallskip

6) Fix a $q > 2 \vee (d-2)$ satisfying the assumptions of the remark. Since $R_\mu^n f=(\mu+\Lambda_{q}(\tilde{a}_n,\nabla \tilde{a}_n + \tilde{b}_n))^{-1}f$, we obtain by \eqref{reg_star} that for all $\mu>\mu_0$
$$
\|\nabla R_\mu^n f\|_{qj} \leqslant K \|f\|_q, \quad j=\frac{d}{d-2}, \quad \mu>\mu_0.
$$
By a weak compactness argument in $L^{qj}$, in view of 1), we have $|\nabla R_\mu^Q f| \in L^{qj}$, and there is a subsequence of $\{R_\mu^nf\}$ (without loss of generality, it is $\{R_\mu^nf\}$ itself) such that
$$
\nabla R_\mu^n f \overset{w}{\longrightarrow} \nabla R_\mu^Q f \quad \text{ in } L^{qj}(\mathbb R^d,\mathbb R^d).
$$
By Mazur's Lemma, there is a sequence of convex combinations of the elements of $\{\nabla R_\mu^n f\}_{n=1}^\infty$ that converges to $\nabla R_\mu^Q f$ strongly in $L^{qj}(\mathbb R^d,\mathbb R^d)$, i.e.
$$
\sum_\alpha c_\alpha \nabla R_\mu^{n_\alpha} f \overset{s}{\longrightarrow } \nabla R_\mu^Q f \quad \text{ in $L^{qj}(\mathbb R^d,\mathbb R^d)$}.
$$
Now, in view of 1), the latter and the Sobolev Embedding Theorem yield $\sum_\alpha c_\alpha R_\mu^{n_\alpha} f \overset{s}{\longrightarrow } R_\mu^Q f$ in $C_\infty$.
Therefore, by 5),  $(\mu+\Lambda_{C_\infty}(a,\nabla a + b))^{-1} f(x) = R_\mu^Q f(x)$ for all $x \in \mathbb R^d$, $f \in C_c^\infty$, as needed.

\end{document}